\documentclass[12pt]{amsart}
\usepackage{amssymb,amsmath,amsfonts,latexsym}
\usepackage{bm}
\usepackage[all,cmtip]{xy}
\usepackage{amscd}

\newcommand{\newsection}[1]{\setcounter{equation}{0} \section{#1}}
\newcommand{\bea}{\begin{eqnarray}}
\newcommand{\eea}{\end{eqnarray}}

\newcommand{\clb}{\mathcal{B}}

\newcommand{\clh}{\mathcal{H}}

\def\textmatrix#1&#2\\#3&#4\\{\bigl({#1 \atop #3}\ {#2 \atop #4}\bigr)}
\def\dispmatrix#1&#2\\#3&#4\\{\left({#1 \atop #3}\ {#2 \atop #4}\right)}
\newcommand{\be}{\begin{equation}}
\newcommand{\ee}{\end{equation}}
\newcommand{\ben}{\begin{eqnarray*}}
\newcommand{\een}{\end{eqnarray*}}

\newcommand{\NI}{\noindent}

\newcommand{\bi}{\begin{itemize}}
\newcommand{\ei}{\end{itemize}}

\newtheorem{Theorem}{\sc Theorem}[section]
\newtheorem{Lemma}[Theorem]{\sc Lemma}

\theoremstyle{definition}

\theoremstyle{plain}

\newtheorem{thm}{Theorem}[section]
\newtheorem{cor}[thm]{Corollary}

\theoremstyle{definition}

\newtheorem{rem}[thm]{Remark}

\numberwithin{equation}{section}

\let\phi=\varphi

\begin{document}

\title[Orlicz extension of Numerical radius inequalities]
{Orlicz extension of Numerical radius inequalities}

\author[Maji]{Amit Maji}
\address{Indian Institute of Technology Roorkee, Department of Mathematics, Roorkee-247 667, Uttarakhand,  India}
\email{amit.maji@ma.iitr.ac.in, amit.iitm07@gmail.com}
\author[Manna]{Atanu Manna$^\dag$ }
\address{Indian Institute of Carpet Technology, Bhadohi-221401, Uttar Pradesh, India}
\email{ atanu.manna@iict.ac.in, atanuiitkgp86@gmail.com ( $^\dag$Corresponding author)}
\author[Mohapatra]{Ram Mohapatra }
\address{University of Central Florida, Department of Mathematics, 4000 Central Florida Blvd, Orlando, USA}
\email{ ram.mohapatra@ucf.edu}



\begin{abstract}
In this paper, we achieve new and improved numerical radius
inequalities of operators defined on a Hilbert space by using Orlicz function and Hermite-Hadamard inequality.
The upper bounds of various inequalities involving numerical radii have been
obtained. Finally, we compute an upper bound of the numerical radius
for block matrices of the form
$\begin{bmatrix}O
& P\\Q & O
\end{bmatrix}$,
where $P, Q$ are any bounded linear operators on a Hilbert space.
\end{abstract}
\maketitle

\newsection{Introduction}

The concepts of numerical range and numerical radius of an operator
play a very significant role and have been studied extensively due
to their enormous applications in engineering, quantum computing,
quantum mechanics, numerical analysis, differential equations, fluid dynamics,
the geometry of Banach spaces, etc (see \cite{AXELSSON}, \cite{BONSALL}, \cite{HORN}).  The study of numerical radius
inequalities and its various improvements are the latest trends
of research in the theory of operators on Hilbert space.

We denote $\clb(\clh)$ as the $C^{*}$-algebra of all bounded linear
operators on a Hilbert space $\clh$. For any bounded linear operator
$A$ on $\clh$ the numerical radius, denoted by $w(A)$, yields a norm on
$\clb(\clh)$. The fundamental inequality between numerical radius and
operator norm of any operator $A\in \mathcal{B}(\mathcal{H})$ is as follows:
\begin{align}\label{opnuinq}
\frac{1}{2}\|A\|\leq w(A) \leq \|A\|.
\end{align}
This says that the operator norm and the norm induced by the numerical radius are
equivalent on $\mathcal{B}(\mathcal{H})$. If $A$ is a normal operator
(that is, $A^*A=AA^*$ ), then $w(A)=\|A\|$ and if $A^2=O$, the zero operator
on $\clh$, then we have $w(A)=\frac{1}{2}\|A\|$. Many researchers later discovered sharp inequalities for a single operator.

Some preliminary results on numerical radius and its
applications to stability theory of finite-difference approximations for
hyperbolic initial value problems were presented by Goldberg and Tadmor
\cite{GOLDON}, (see also \cite{GOLDNUM}). The study of power inequalities
for numerical range,
and numerical radius (that is, $w(A^n)\leq w(A)^n$, $n\in \mathbb{N}$,
$A\in \mathcal{B}(\mathcal{H})$) were studied in \cite{BERG} and
\cite{PEAR} (see also \cite{HALMOS}), respectively. The first refinement
of inequality (\ref{opnuinq}) was obtained by Kittaneh \cite{KITTSM03}. We mention some results as below:
\begin{align}\label{kitt03}
w(A) \leq \frac{1}{2}(\|A\|+\|A^2\|^{1/2}).
\end{align}
In 2005, Kittaneh \cite{KITTSM05} further improved the inequality
(\ref{opnuinq}) as follows: For $A\in \mathcal{B}(\mathcal{H})$
\begin{align}\label{kitt05}
\frac{1}{4}\||A|^2+|A^*|^2\|\leq w^2(A)\leq \frac{1}{2} \||A|^2+|A^*|^2\|.
\end{align}
Further improvement of the upper bound of inequality (\ref{kitt05})
was provided by El-Haddad and Kittaneh \cite{HADDKITT} and they established
the following:\\
For $A\in \clb(\clh)$ and $r\geq 1$
\begin{align}{\label{hadkitt}}
w^r(A)\leq \frac{1}{2} \||A|^{2r\alpha}+|A^*|^{2r(1-\alpha)}\| \nonumber\\
\mbox{and}~~w^{2r}(A)\leq \|\alpha|A|^{2r} + (1-\alpha)|A^*|^{2r}\|,
\end{align}where $0\leq \alpha \leq 1$.
Further refinement of the upper bounds of both the inequalities
(\ref{kitt03}) and (\ref{kitt05}) was obtained by Abu-Omar and Kittaneh
\cite{OMARKITT} and their result is stated as below:
\begin{align}{\label{omkitt}}
w^2(A)\leq \frac{1}{4}\||A|^{2}+|A^*|^{2}\|+ \frac{1}{2}w(A^2).
\end{align}
A stronger version of the inequalities (\ref{kitt03}) and (\ref{kitt05})
for the upper bounds is recently obtained by Bhunia and Paul \cite{BHUKAL}:\\
For $A\in \mathcal{B}(\mathcal{H})$ and any $r\geq 1$
\begin{align}{\label{bhupaul}}
w^{2r}(A)\leq \frac{1}{4}\||A|^{2r}+|A^*|^{2r}\|+ \frac{1}{2}w(|A|^r|A^*|^r).
\end{align}
The above results show a continuous stream of research on the numerical radius inequalities and its recent improvements and
generalizations for a single operator $A\in \mathcal{B}(\mathcal{H})$.
If $A, B\in \mathcal{B}(\mathcal{H})$, then Holbrook \cite{HOLB}
proved that $w(AB)\leq 4w(A)w(B)$. In addition, if $A$ and $B$
commute, then $w(AB)\leq 2w(A)w(B)$, where the constant $2$ is
sharp. Also the following result is due to Fong and Holbrook
\cite{FONGHOLB} is worth mentioning:
\begin{align*}
w(AB+BA) \leq 2\sqrt{2}w(A)\|B\|.
\end{align*}
Kittaneh \cite{KITTSM05} established a generalized
numerical radius inequality: Let $A, B, C, D, S$, and
$T \in \mathcal{B}(\mathcal{H})$. Then
\begin{align}
w(ATB+CSD)\leq \frac{1}{2}\|A|T^*|^{2(1-\alpha)}A^*+
B^*|T|^{2\alpha}B + C|S^*|^{2(1-\alpha)}C^*+D^*|S|^{2\alpha}D\|
\end{align}
for $0 \leq \alpha \leq 1$.
In \cite{HIRJKITT}, and \cite{HIRJKITT2}, the authors computed
the numerical radius of product of operators, namely,
$w(A^*XA)$, $w(A^*XB\pm B^*YA)$ in terms of operator norm and
numerical radius of a block matrix of the form
$\begin{bmatrix}
O & P\\
Q & O
\end{bmatrix}$.
There are many recent research papers (e.g. \cite{OMARKITT},
\cite{BHUKAL}, \cite{DRAGO}, \cite{SABA}, \cite{YAMA} and references cited
therein) on the numerical radius inequalities with its
improvements and generalizations.

In this article, we attempt to unify as well as extend the above
numerical radius inequalities for a large class of operators. We
also establish some new numerical
radius inequalities with improvements of many known results. Further, we provide a general
upper bound for numerical radius of block matrices. The main guiding
tools used in this article are the various operator inequalities,
the notion of functional calculus, and the Orlicz function.
\vspace{.1cm}

The paper is organized as follows: In Section 2, we discuss
preliminaries and some important inequalities. Section 3 deals with
various numerical radius inequalities of Hilbert space operators using
Orlicz functions and Hermite-Hadamard inequality illustrated with some concrete examples on the Hardy space.
In Section 4, we discuss the upper bound of numerical radius for the off-diagonal
block matrix of operators.

\newsection{Preliminaries}

Throughout this paper, $\clh$ denotes a separable complex Hilbert space
and $\clb(\clh)$ is the algebra of all bounded linear operators from $\clh$
to itself. For any $A\in \mathcal{B}(\mathcal{H})$, $A^*$ denotes the
adjoint of $A$ and $|A|:=(A^*A)^{1 \over 2}$. We call an operator $A$ on a Hilbert space $\clh$ is an isometry if $A^*A=I$. \\
We say an operator $A\in \mathcal{B}(\mathcal{H})$ is positive if $A$ is self adjoint and
$\langle Ax, x \rangle \geq 0$ for all $x \in \clh$. The operator norm,
denoted by $\|A\|$, is defined as
\[
\|A\|=\sup\{\|Ax\|: ~~x\in \clh~\mbox{with}~\|x\|=1\}.
\]
We write the cartesian decomposition of $A$ by $A=\mathfrak{R}(A)+
i \mathfrak{I}(A)$, where $\mathfrak{R}(A)=\frac{1}{2}(A+A^*)$ and
$\mathfrak{I}(A)=\frac{1}{2i}(A-A^*)$. For any $A\in \mathcal{B}(\mathcal{H})$,
the numerical range of $A$, denoted by $W(A)$, is defined as
\[
W(A) :=\{\langle Ax,x \rangle : x\in \clh ,\|x\|=1 \}.
\]
Then the numerical radius of $A$, denoted by $w(A)$, is given by
\[
w(A)=\sup\{ |\langle Ax,x \rangle| : x\in \clh ,\|x\|=1 \}.
\]
Similarly, the Crawford number of $A$ is denoted by $c(A)$, and is defined as below:
\begin{center}
$c(A)=\inf\{ |\langle Ax,x \rangle| : x\in \clh ,\|x\|=1 \}.$
\end{center}

Before proceeding further, let us recall the notion of functional
calculus given in \cite{Conway-Book} and the Orlicz function. Let $A$ be a normal
operator on $\clh$ and $\sigma(A)$ denote the spectrum of $A$. Let
$B_{\infty}(\sigma(A))$ and $\mathcal{C}(\sigma(A))$ be
the bounded measurable complex-valued functions and
the continuous complex-valued functions on $\sigma(A)
\subseteq \mathbb{C}$, respectively. Clearly, $B_{\infty}(\sigma(A))$
is a $C^*$-algebra with the involution map defined by $f \mapsto \bar{f}$.
Let $\pi: \mathcal{C}(\sigma(A)) \rightarrow \clb(\clh)$ be a $*$-homomorphism
such that $\pi(1)= I_{\clh}$, where $1$ is a constant function with value one
and $I_{\clh}$ is an identity operator on $\clh$. Then there exists a unique spectral measure
$\mathcal{P}$ in $(\sigma(A), \clh)$ such that
\[
\pi(f) = \int_{\sigma(A)}fd\mathcal{P}.
\]
In particular, if $A \in \clb(\clh)$ is a positive operator, then
$\sigma(A)$ is a subset of $[0, \infty)$ and $f(A)$ is a positive
operator for any positive continuous function whose domain contains the
spectrum of $A$.\\

\NI
A map $\varphi: [0, \infty]\rightarrow [0, \infty]$ is said to be
an \emph{Orlicz function (non-degenerate)} if it is convex, continuous, non-decreasing,
with $\varphi(0)=0$, $\varphi(u)>0$ when $u>0$ and $\varphi(u)\rightarrow\infty$
as $u\rightarrow\infty$ (see Definition 4.a.1., \cite{LINDTZA}). An Orlicz function is said to be \emph{degenerate} if $\varphi(u)=0$ for some $u>0$. Throughout the paper wherever Orlicz function appears, we mean it \emph{non-degenerate}. Some concrete examples of Orlicz
functions are $(i)$ $\varphi(u)=u^p$, $p\geq 1$, $(ii)$ $\varphi_r(u)=e^{u^r}-1$, $1<r<\infty$
$(iii)$ $\varphi_p(u)=\frac{u^p}{\ln(e+u)}$, $p\geq 2$
and many more. An Orlicz function can be expressed
by the following integral representation:
\begin{center}
$\varphi(u)=\displaystyle\int_{0}^{u}p(t) dt$,
\end{center}where $p$ is a non-decreasing function such that $p(0)=0$, $p(t)>0$ for $t>0$, $\displaystyle\lim_{t\rightarrow\infty}p(t)=\infty$, and is known as kernel of $\varphi$. These restrictions exclude the case only when $\varphi(u)$ is equivalent to the function $u$. The right inverse $q$ of $p$ is defined as $q(s)=\sup\{t: p(t)\leq s\}$, $s\geq 0$. Then $q$ satisfy similar properties as $p$. Using it following complementary Orlicz function $\psi$ to
$\varphi$ is defined as below:
\begin{center}
$\psi(v)=\displaystyle\int_{0}^{v}q(s) ds$.
\end{center}
The pair $(\varphi, \psi)$ is called mutually complementary Orlicz functions.\\
The following inequalities are useful in the sequel. First, we recall the well known Hermite-Hadamard inequality \cite{HADAMARD} for a convex function $\varphi: I\subseteq \mathbb{R} \rightarrow\mathbb{R}$. Suppose that $a, b\in I$ such that $a<b$. Then Hermite-Hadamard inequality states that
\begin{align}{\label{HERMITEHAD}}
\varphi \Big(\frac{a+b}{2}\Big)&\leq \displaystyle\int_{0}^{1}\varphi(ta+(1-t)b)dt\leq \frac{\varphi(a)+\varphi(b)}{2},
\end{align}which is used to find the sharp inequalities.
\begin{Lemma}(See Theorem 1.2, \cite{PECA}){\label{opjensen}}
Let $A$ be any self-adjoint operator on a Hilbert space $\clh$
and $\varphi$ be a convex function such that the spectrum
of $A$ contained in $[0, \infty]$. Then the operator version
of Jensen inequality states that
\begin{center}
$\varphi(\langle Ax, x \rangle) \leq \langle \varphi(A)x, x \rangle \quad (\mbox{for every unit vector}~ x \in \clh)$.
\end{center}
\end{Lemma}
An easy consequence of Lemma \ref{opjensen} is the H\"{o}lder-McCarthy inequality, which is stated as below:
\begin{Lemma}(See Theorem 1.4, \cite{PECA}){\label{mccarthy}}
Let $A$ be any positive operator on a Hilbert space $\clh$ and $r \geq 1$
be any real number. Then for any unit vector $x \in \clh$,
\begin{center}
$\langle Ax, x \rangle^r \leq \langle A^{r}x, x \rangle$.
\end{center}
\end{Lemma}

\begin{Lemma} \emph{(}\cite{KITTNOTES}, Theorem 1\emph{)}{\label{kitt88}}
Let $A\in \clb(\clh)$, $f$, $g$ be non-negative continuous functions on $[0, \infty]$ such that $f(t)g(t)=t$ for all $t\in [0, \infty]$. Then
\begin{center}
$|\langle Ax, y \rangle| \leq \|f(|A|)x\|\|g(|A^*|)y\|$ holds
\end{center}for all $x, y\in \clh$.
\end{Lemma}

In particular, if $f(t)=t^\alpha$ and $g(t)=t^{2(1-\alpha)}$ for $0\leq \alpha \leq 1$,
then from Lemma \ref{kitt88} one gets the following special form.
\begin{Lemma}{\label{mixedcs}}\cite{KATO}
Let $A$ be any bounded linear operator on a Hilbert space $\clh$.
Then for $x, y\in \clh$
\begin{center}
$|\langle Ax, y \rangle|^2 \leq \langle |A|^{2\alpha}x, x \rangle
\langle |{A^*}|^{2(1-\alpha)}y, y \rangle$,
\end{center}
where $0\leq \alpha \leq 1$.
\end{Lemma}

\begin{Lemma}{\label{buzano}}\cite{BUZANO}
Suppose that $x, y, e \in \clh$ such that $\|e\|=1$. Then the following holds:
\begin{center}
$|\langle x, e \rangle \langle e, y \rangle| \leq \frac{1}{2}(\|x\|\|y\|+|\langle x, y \rangle|)$.
\end{center}
\end{Lemma}

\begin{Lemma}\cite{LINDTZA}(Young's inequality){\label{young}}
Let $\varphi$, and $\psi$ be two complementary Orlicz functions. Then \\
$(i)$ for $u, v\geq 0$, $uv\leq \varphi(u)+\psi(v)$, and\\
$(ii)$ for $u\geq 0$, $up(u)=\varphi(u)+\psi(p(u))$ (equality condition).
\end{Lemma}

\begin{Lemma}{\label{lembohr}}
Let $\varphi$ be an Orlicz function and suppose that $a_i\geq 0$ for $i=1, 2,\ldots, n$. Then
$\varphi(\displaystyle\sum_{i=1}^{n}\frac{a_i}{n})\leq \frac{1}{n}\sum_{i=1}^{n}\varphi(a_i).$
\end{Lemma}
\begin{proof}
Let $\alpha\in [0, 1]$ and $\varphi$ be an Orlicz function. Then using non-decreasing property of the kernel function $p(t)$, we have the following:
\begin{align*}
\varphi(\alpha u)&=\displaystyle\int_{0}^{\alpha u}p(t) dt = \alpha \displaystyle\int_{0}^{u}p(\alpha v) dv\leq \alpha \displaystyle\int_{0}^{u}p(v) dv=\alpha \varphi(u),
\end{align*}which is an important and well-known property of $\varphi$ (see \cite{LINDTZA}) and we have recalled it here for the sake of completeness. Hence the desired inequality follows by choosing $\alpha$ and $u$ suitably.
\end{proof}

\begin{rem} (see \cite{VASIKEC})
An easy consequence of Lemma \ref{lembohr} is the Bohr's inequality, which says that for $a_i\geq 0$ for $i=1, 2,\ldots, n$ and $\varphi(t)=t^r$, $r\geq 1$ we have $$(\displaystyle\sum_{i=1}^{n}\frac{a_i}{n})^r\leq \frac{1}{n}\sum_{i=1}^{n}a_i^r.$$
\end{rem}

\section{Main Results}

In this section, we unify the earlier results on numerical radius inequalities
as well as extend the results for a larger class of Hilbert space operators
using Orlicz function and functional calculus. This also yields simple proofs of
many of the earlier results. Let us begin with the following result.
\begin{thm}{\label{strngkit}}
Suppose that $A\in \clb(\clh)$. Then for $0\leq \alpha, t \leq 1$ and any Orlicz function $\varphi$
$$\varphi(w(A))\leq \Big\|\displaystyle\int_{0}^{1}\varphi(t|A|^{2\alpha}+(1-t)|A^*|^{2(1-\alpha)})dt\Big\|\leq\frac{1}{2}\|\varphi(|A|^{2\alpha}) + \varphi(|A^*|^{2(1-\alpha)})\|.$$
\end{thm}
\begin{proof}
Choose $x\in \clh$ such that $\|x\|=1$. Given $\alpha, t\in [0, 1]$. By using \mbox{Lemma}~ \ref{opjensen}, \mbox{Lemma}~ \ref{mixedcs}, \emph{AM-GM} inequality, and inequality (\ref{HERMITEHAD}) we have the following:
\begin{align*}
\varphi(|\langle Ax, x \rangle|)
&\leq \varphi(\langle |A|^{2\alpha}x, x \rangle^{\frac{1}{2}} \langle |A^*|^{2(1-\alpha)}x, x \rangle^{\frac{1}{2}})
\leq \varphi\Big(\frac{\langle |A|^{2\alpha}x, x \rangle+\langle |A^*|^{2(1-\alpha)}x, x \rangle}{2}\Big)\\
& \leq \displaystyle\int_{0}^{1}\varphi(t\langle|A|^{2\alpha} x, x\rangle +(1-t)\langle|A^*|^{2(1-\alpha)}x, x\rangle)dt\\
&= \displaystyle\int_{0}^{1}\varphi\langle (t|A|^{2\alpha}+(1-t)|A^*|^{2(1-\alpha)})x, x\rangle dt\\
& \leq \displaystyle\int_{0}^{1}\langle \varphi(t|A|^{2\alpha}+(1-t)|A^*|^{2(1-\alpha)})x, x\rangle dt\\
& = \Big\langle \displaystyle\int_{0}^{1}\varphi(t|A|^{2\alpha}+(1-t)|A^*|^{2(1-\alpha)}) dt~~ x, x\Big \rangle.
\end{align*}
Since $\varphi$ is convex, so we have $\varphi(t|A|^{2\alpha}+(1-t)|A^*|^{2(1-\alpha)})\leq t\varphi(|A|^{2\alpha})+(1-t)\varphi(|A^*|^{2(1-\alpha)})$. Therefore
\begin{align*}
\varphi(|\langle Ax, x \rangle|) & \leq \Big\langle \displaystyle\int_{0}^{1}\varphi(t|A|^{2\alpha}+(1-t)|A^*|^{2(1-\alpha)}) dt~~ x, x\Big \rangle & \leq \frac{1}{2}\langle \{\varphi(|A|^{2\alpha})+ \varphi(|A^*|^{2(1-\alpha)})\}x, x \rangle.
\end{align*}Taking the supremum over $\|x\|=1$, we get the desired inequality.
\end{proof}
An immediate consequence of the above result is an improvement of the Haddad and Kittaneh inequality \cite{HADDKITT}.
\begin{cor}Suppose $A\in \clb(\clh)$, $\varphi(t)=t^r$ for $t\geq 0$ and $r\geq 1$.
Then for $0\leq \alpha\leq 1$ $$w^r(A)\leq \Big\|\displaystyle\int_{0}^{1}\Big(t|A|^{2\alpha}+(1-t)|A^*|^{2(1-\alpha)}\Big)^r dt\Big\|\leq \frac{1}{2} \||A|^{2r\alpha}+|A^*|^{2r(1-\alpha)}\|.$$
\end{cor}
\begin{rem}
The above findings are stronger versions of celebrated Kittaneh inequalities (see \cite{HADDKITT}, \cite{KITTSM03}, and \cite{KITTSM05}) as well as results of Sababheh and Moradi \cite{SABA}.
\end{rem}

\begin{rem}
Suppose that $A$ is an isometry on a Hilbert space $\clh$ and $\varphi(t)=t^r$, $r\geq 1$. Then from the above result we have $1=\varphi(1)\leq \frac{1}{2}\|I+P\|\leq 1$, that is $\|I+P\|=2$, where $P$ is the orthogonal projection on the range of $A$.
\end{rem}

\begin{rem}
Let $H^2(\mathbb{D})$ be the Hardy space over the unit disc $\mathbb{D}$ and $\theta$ be any inner function on $\mathbb{D}$. Then $T_\theta$ is an analytic Toeplitz operator on $H^2(\mathbb{D})$. Now we have from our result $w(T_\theta)=1$ (see \cite{MAJEE}). Then for $\varphi(t)=t^r$, $r\geq 1$, we get $\|I+P_{\theta H^2(\mathbb{D})}\|=2$, where $P_{\theta H^2(\mathbb{D})}$ is the orthogonal projection onto the range of $T_\theta$, that is $\theta H^2(\mathbb{D})$.
In a similar way, in case of Crawford number, we have $\varphi(c(A))\leq \frac{1}{2}\|\varphi(|A|^{2\alpha}) + \varphi(|A^*|^{2(1-\alpha)})\|$, $0\leq \alpha \leq 1$. In this case, $c(T_\theta)=0$ (see \cite{MAJEE}), hence $\varphi(c(T_\theta))=0$, for any Orlicz function. Then for $\varphi(t)=t^r$, $r\geq 1$, we get $0\leq \|I+P_{\theta H^2(\mathbb{D})}\|\leq 2$.
\end{rem}

\begin{thm}
Let $A\in \clb(\clh)$ and $\varphi$ be an Orlicz function. Then
\[w^2(A)\leq \|\varphi(|A|)+\psi(|A^*|)\|,\]
\end{thm}where $\psi$ is a complementary Orlicz function to $\varphi$.
\begin{proof}
Suppose that $A\in \clb(\clh)$. Then for any $x\in \clh$ with $\|x\|=1$, we have
\begin{align*}
w^2(A) & =\displaystyle \sup_{\|x\|=1} |\langle Ax, x \rangle|^2 \\
& \leq \displaystyle \sup_{\|x\|=1} \langle |A|x, x \rangle \langle |A^*|x, x \rangle ~~(\mbox{Lemma}~ \ref{mixedcs})\\
& \leq \displaystyle \sup_{\|x\|=1} \{\varphi(\langle |A|x, x \rangle) + \psi(\langle |A^*|x, x \rangle)\} ~~(\mbox{Lemma}~ \ref{young})\\
& \le \displaystyle \sup_{\|x\|=1}\langle\{\varphi(|A|) + \psi(|A^*|)\}x, x \rangle
~~(\mbox{Lemma}~ \ref{opjensen})\\
& = w(\varphi(|A|) + \psi(|A^*|)) \\
& \leq \|\varphi(|A|)+\psi(|A^*|)\|.
\end{align*}
\end{proof}

Now we have the following inequalities in particular choices of Orlicz
function $\varphi$.
\begin{cor} Let $A\in \clb(\clh)$ and $\varphi(t)=\frac{t^p}{p}$ for $p>1$ and $t\geq 0$. Then the corresponding complementary Orlicz function is $\psi(t)=\frac{t^q}{q}$, where $\frac{1}{p}+\frac{1}{q}=1$, and \\
$(i)$ $w^2(A)\leq w(\frac{|A|^p}{p}+\frac{|A^*|^q}{q})\leq \|\frac{|A|^p}{p}+\frac{|A^*|^q}{q}\|$,\\
$(ii)$ $w^2(A)\leq \|\frac{|A|^2}{2}+\frac{|A^*|^2}{2}\|$, when $p=2$ (Kittaneh \cite{KITTSM05}).
\end{cor}

\begin{rem}
Suppose that $A$ is an isometry on a Hilbert space $\clh$. Then we have $w(A)=1$ (see \cite{MAJEE}). Also $|A|=\sqrt{A^*A}=I$ and $|A^*|=\sqrt{AA^*}=P$, where $P$ is the orthogonal projection on the range of $A$. Then from the above result with $\varphi(t)=\frac{t^p}{p}$, $\psi(t)=\frac{t^q}{q}$ for $p, q>1$, we get
\begin{center}
$\|\frac{1}{p}I+\frac{1}{q}P\|=1$.
\end{center}
\end{rem}

In our next result, we obtain an upper bound for numerical radius as below.
\begin{thm}{\label{thmastarxbphpsi}}
Let $A, B, X\in \clb(\clh)$, and $\varphi$, $\psi$ be two mutually complementary Orlicz functions. Then\\
(i) $w^r(A^*XB)\leq \|X\|^r w(\varphi(|A|^r)+\psi(|B^r|))$ for $r\geq 2$, and \\
(ii) $w(A^*XB)\leq  \varphi(\sqrt{w(B^*|X|^{2\alpha}B)})+\psi(\sqrt{w(A^*|X|^{2(1-\alpha)}A)})$
for $0\leq \alpha \leq 1$.
\end{thm}
\begin{proof}
(i) Suppose that $A, B, X\in \clb(\clh)$. Let $x\in \clh$ with $\|x\|=1$ and $r \geq 2$.
Then we compute the following expression
\begin{align*}
|\langle A^*XB x, x\rangle|^r
&= |\langle XB x, Ax\rangle|^r\\
& \leq \|X\|^r\|Ax\|^r\|Bx\|^r\\
&= \|X\|^r\langle Ax, Ax \rangle^{\frac{r}{2}}\langle Bx, Bx\rangle^{\frac{r}{2}}\\
& \leq \|X\|^r\langle |A|^r x, x\rangle\langle |B|^r x, x\rangle~~(\mbox{Lemma}~ \ref{mccarthy})\\
& \leq \|X\|^r (\varphi(\langle |A|^r x, x\rangle)+ \psi(\langle |B|^r x, x\rangle))\\
\end{align*}
\begin{align*}
& \leq \|X\|^r (\langle \varphi(|A|^r) x, x\rangle + \langle \psi(|B|^r) x, x\rangle)~~(\mbox{Lemma}~ \ref{opjensen})\\
& \leq \|X\|^r w(\varphi(|A|^r)+\psi(|B^r|)).
\end{align*}Hence the result follows.\\
(ii) Further, applying \mbox{Lemma}~ \ref{mixedcs} we have
\begin{align*}
|\langle A^*XB x, x\rangle|
& = |\langle XB x, Ax\rangle|\\
& \leq \langle |X|^{2\alpha}B x, Bx\rangle^{\frac{1}{2}}\langle |X|^{2(1-\alpha)}A x, Ax\rangle^{\frac{1}{2}} \\
& \leq \varphi(\langle B^*|X|^{2\alpha}Bx, x\rangle^{\frac{1}{2}}) + \psi(\langle A^*|X|^{2(1-\alpha)}A x, x\rangle^{\frac{1}{2}}),
\end{align*}
where $0\leq \alpha \leq 1$.
This proves the desired inequality.
\end{proof}

Now we establish upper bounds for $\varphi(w^2(A))$.
\begin{thm}
For any Orlicz function $\varphi$, and $A\in \clb(\clh)$ the following inequalities hold: \\
(i) $\varphi(w^2(A))\leq \|\alpha\varphi(|A|^{\frac{1}{\alpha}})+(1-\alpha)\varphi(|A^*|^{\frac{1}{1-\alpha}})\|$ for all $\alpha\in (0, 1)$, and\\
(ii) $\varphi(w^2(A))\leq \|\alpha\varphi(|A|^2) + (1-\alpha)\varphi(|A^*|^2)\|$ for all $\alpha\in (0, 1)$.
\end{thm}
\begin{proof}
(i) Let $A\in \clb(\clh)$, $\alpha\in(0, 1)$ and choose $x\in \clh$ such that $\|x\|=1$. Then by \mbox{Lemma}~ \ref{mixedcs} and Young's inequality, we get
\begin{align*}
\varphi(|\langle Ax, x \rangle|^2)
&\leq \varphi(\langle |A|x, x \rangle \langle |A^*|x, x \rangle)\\
& \leq \varphi(\alpha \langle |A|x, x \rangle^{\frac{1}{\alpha}}+ (1-\alpha)\langle |A^*|x, x \rangle^{\frac{1}{1-\alpha}})\\
& \leq \alpha \langle \varphi(|A|^{\frac{1}{\alpha}})x, x \rangle + (1-\alpha)\langle \varphi(|A^*|^{\frac{1}{1-\alpha}})x, x \rangle~~(\mbox{convexity~ of~}\varphi~~\mbox{and}~\mbox{Lemma}~ \ref{opjensen})\\
& = \langle \{\alpha\varphi(|A|^{\frac{1}{\alpha}})+ (1-\alpha)\varphi(|A^*|^{\frac{1}{1-\alpha}})\}x, x \rangle,
\end{align*}which proves the desired inequality.\\
(ii) Using \mbox{convexity~ of~}$\varphi$, ~~\mbox{Lemma}~ \ref{opjensen}, ~~\mbox{Lemma}~ \ref{mccarthy}, and Young's inequality it follows that:
\begin{align*}
&\varphi(|\langle Ax, x \rangle|^2)\\
&\leq \varphi(\langle |A|^{2\alpha}x, x \rangle \langle |A^*|^{2(1-\alpha)}x, x \rangle)~~(\mbox{Lemma}~ \ref{mixedcs})\\
& \leq \varphi(\alpha \langle |A|^{2\alpha}x, x \rangle^{\frac{1}{\alpha}}+ (1-\alpha)\langle |A^*|^{2(1-\alpha)}x, x \rangle^{\frac{1}{1-\alpha}})~~\\
& \leq \alpha \langle \varphi(|A|^{2})x, x \rangle + (1-\alpha)\langle \varphi(|A^*|^{2})x, x \rangle\\
& = \langle \{\alpha\varphi(|A|^{2})+ (1-\alpha)\varphi(|A^*|^{2})\}x, x \rangle,
\end{align*} which gives the desired inequality.
\end{proof}
Now we obtain a general form of inequality (\ref{omkitt}) established by Abu-Omar and Kittaneh (see \cite{OMARKITT}). In fact we have the following theorem.
\begin{thm}{\label{improvmkitt}}
Let $\varphi$ be an Orlicz function and $A\in \clb(\clh)$. Then
$$\varphi(w^2(A))\leq \Big\|\displaystyle\int_{0}^{1}\varphi(t|A|^{2}+(1-t)|A^*|^{2})dt \Big\|+ \frac{1}{2}\varphi(w(A^2))\leq \frac{1}{4}\|\varphi(|A|^{2})+\varphi(|A^*|^{2})\|+ \frac{1}{2}\varphi(w(A^2)).$$
\end{thm}
\begin{proof}
Suppose that $A\in \clb(\clh)$ and $x\in \clh$ with $\|x\|=1$. Then \mbox{Lemma}~ \ref{mixedcs}, \mbox{Lemma}~ \ref{buzano}, \emph{AM-GM}~ inequality, and convexity of $\varphi$ implies that
\begin{align}{\label{improvmkittinq1}}
&\varphi(|\langle Ax, x \rangle|^2)\nonumber\\
& = \varphi(|\langle Ax, x \rangle \langle x, A^*x \rangle|)\nonumber\\
&\leq \varphi(\frac{1}{2}(\|Ax\|\|A^*x\| + |\langle Ax, A^*x \rangle|))\nonumber\\
& \leq \frac{1}{2}\varphi(\langle Ax, Ax \rangle^{\frac{1}{2}}\langle A^*x, A^*x \rangle^{\frac{1}{2}})+\frac{1}{2}\varphi(|\langle A^2x, x \rangle|)\nonumber\\
& \leq \frac{1}{2}\varphi(\frac{\langle A^*Ax, x \rangle + \langle AA^*x, x \rangle}{2})+ \frac{1}{2}\varphi(|\langle A^2x, x \rangle|).
\end{align}
Using the Hermite-Hadamard inequality (\ref{HERMITEHAD}) in (\ref{improvmkittinq1}), we get
\begin{align}
&\varphi(|\langle Ax, x \rangle|^2)\nonumber\\
& \leq \frac{1}{2}\Big\langle\displaystyle\int_{0}^{1} \{\varphi(t|A|^{2}+(1-t)|A^*|^{2})dt\}x, x\Big\rangle+ \frac{1}{2}\varphi(|\langle A^2x, x \rangle|) \nonumber\\
& \leq  \frac{1}{4}(\langle \varphi(|A|^{2})x, x \rangle + \langle \varphi(|A^*|^{2})x, x \rangle)+ \frac{1}{2}\varphi(|\langle A^2x, x \rangle|)~~(\mbox{Lemma}~ \ref{opjensen}).\nonumber
\end{align}
This proves the desired inequality.
\end{proof}

\begin{rem}
Let $T_z$ be the shift operator on the $E$-valued Hardy space $H_E^2(\mathbb{D})$ over the unit disc $\mathbb{D}$. Then $w^2(T_z)=1$ (see \cite{MAJEE}). Therefore by choosing $\varphi(t)=t^r$, $r\geq 1$ in the above result, we have $1\leq \frac{1}{4}\|I+P\|+\frac{1}{2}$, that is $\|I+P\|=2$, where $P$ is the orthogonal projection onto $zH_E^2(\mathbb{D})$.
\end{rem}

Sometimes it is very difficult to obtain the complementary Orlicz function $\psi$ to $\varphi$ explicitly. So here we obtain an upper bound of numerical radius, which involves only $\varphi$.
\begin{thm}
Let $A, B, X\in \clb(\clh)$ and $\varphi$ be any Orlicz function. Then
$$\varphi(w(A^*XB))\leq \Big\|\displaystyle\int_{0}^{1}\varphi(t\|X\||A|^{2}+(1-t)\|X\||B|^{2})dt \Big\| \leq \frac{1}{2}w(\varphi(\|X\||A|^2)+\varphi(\|X\||B|^2)).$$
\end{thm}
\begin{proof}
Let $x\in \clh$ with $\|x \|=1$. Since $\varphi$ is convex so by applying \mbox{Lemma}~ \ref{opjensen}, one gets
\begin{align*}
\varphi(|\langle A^*XB x, x\rangle|)
&= \varphi(|\langle XBx, Ax\rangle|)\\
& \leq \varphi(\|X\|\|Ax\|\|Bx\|)\\
&= \varphi(\|X\|\langle |A|^2x, x\rangle^{\frac{1}{2}}\langle |B|^2 x, x\rangle^{\frac{1}{2}})\\
& \leq \varphi\Big(\frac{\|X\|\langle |A|^2x, x\rangle+\langle |B|^2 x, x\rangle}{2}\Big).
\end{align*}
Using the convexity of $\varphi$, and \mbox{Lemma}~ \ref{opjensen} we get
\begin{align*}
&\varphi(|\langle A^*XB x, x\rangle|)\\
& \leq \Big\langle\displaystyle\int_{0}^{1}\{\varphi(t\|X\||A|^{2}+(1-t)\|X\||B|^{2})dt\}x, x \Big\rangle\\
& \leq \frac{1}{2}\langle(\varphi(\|X\||A|^2)+\varphi(\|X\||B|^2)) x, x\rangle.
\end{align*}
The result is immediately follows by taking supremum over $\|x \|=1$.
\end{proof}
\begin{rem}
Suppose $X \in \clb(\clh)$ is a contraction, that is, $\|X\|\leq 1$.
Then using the property $\varphi(\alpha t)\leq \alpha \varphi(t)$ for $0<\alpha \leq 1$, $t\geq 0$, we get
$\varphi(\|X\||A|^2)\leq \|X\|\varphi(|A|^2)$, and similarly $\varphi(\|X\||B|^2)\leq \|X\|\varphi(|B|^2)$. Therefore, the above result yields the following inequality for a contraction $X$:
$$\varphi(w(A^*XB))\leq \|X\|\Big\|\displaystyle\int_{0}^{1}\varphi(t|A|^{2}+(1-t)|B|^{2})dt \Big\|\leq \frac{\|X\|}{2}w(\varphi(|A|^2)+\varphi(|B|^2)).$$
\end{rem}
For $X\in \clb(\clh)$, and two positive operators $A, B \in \clb(\clh)$ the operator $A^{\alpha}XB^{1-\alpha}$ for $0 \leq \alpha \leq 1$ has its special importance in operator theory. Here we establish an upper bound of $\varphi(w(A^{\alpha}XB^{1-\alpha}))$ for Orlicz function $\varphi$. Subsequently, we compute $w^r(A^{\alpha}XB^{1-\alpha})$ for $r\geq 1$ and many others by choosing $\varphi$ suitably.
\begin{thm}{\label{thmaalphaxb}}
Let $A, B, X\in \clb(\clh)$ such that $A, B \geq 0$. \\
(i) For any Orlicz function $\varphi$ and $\alpha\in (0, \frac{1}{2}]$
$$\varphi(w(A^{\alpha}XB^{1-\alpha}))\leq w(\alpha\varphi(\|X\|A)+(1-\alpha)\varphi(\|X\|B)).$$
(ii) Further, if $\alpha\in [0, 1]$ then
$$\varphi(w(A^{\alpha}XB^{1-\alpha}))\leq \frac{1}{2}w(\varphi(\|X\|A^{2\alpha})+\varphi(\|X\|B^{2(1-\alpha)})).$$
\end{thm}

\begin{proof}
(i) Let $x\in \clh$ with $\|x \|=1$, and $A, B, X\in \clb(\clh)$
such that $A, B \geq 0$. Using \mbox{Lemma}~ \ref{mixedcs},
Young's inequality and \mbox{Lemma}~ \ref{mccarthy}, we get
\begin{align*}
\varphi(|\langle A^{\alpha}XB^{1-\alpha} x, x\rangle|)
&= \varphi(|\langle XB^{1-\alpha}x, A^{\alpha}x\rangle|)\\
& \leq \varphi(\|X\|\|A^{\alpha}x\|\|B^{1-\alpha}x\|)\\
&= \varphi(\|X\|\langle A^{2\alpha}x, x\rangle^{\frac{1}{2}}\langle B^{2(1-\alpha)} x, x\rangle^{\frac{1}{2}})\\
& \leq \varphi(\|X\|(\alpha\langle A^{2\alpha}x, x\rangle^{\frac{1}{2\alpha}}+ (1-\alpha)\langle B^{2(1-\alpha)} x, x\rangle^{\frac{1}{2(1-\alpha)}}))\\
& \leq \varphi\Big(\|X\|(\alpha\langle Ax, x\rangle + (1-\alpha)\langle B x, x\rangle)\Big)\\
& \leq \langle(\alpha\varphi(\|X\|A)+ (1-\alpha)\varphi(\|X\|B)) x, x\rangle~~(\mbox{since~}\varphi ~\mbox{is~convex}).
\end{align*}This proves the desired inequality.\\
(ii) Proceeding as in the case (i), we get the following inequality
\begin{align*}
&\varphi(|\langle A^{\alpha}XB^{1-\alpha} x, x\rangle|)\\
&= \varphi(\|X\|\langle A^{2\alpha}x, x\rangle^{\frac{1}{2}}\langle B^{2(1-\alpha)} x, x\rangle^{\frac{1}{2}})~~(\mbox{Lemma}~ \ref{mixedcs})\\
& \leq \varphi(\|X\|(\frac{1}{2}\langle A^{2\alpha}x, x\rangle+ \frac{1}{2}\langle B^{2(1-\alpha)} x, x\rangle))\\
& \leq \frac{1}{2}\langle(\varphi(\|X\|A^{2\alpha})+ \varphi(\|X\|B^{2(1-\alpha)})) x, x\rangle. (~\mbox{Lemma}~ \ref{opjensen})
\end{align*}
This proves the desired result.
\end{proof}
\begin{rem}
In addition, if $X\in \clb(\clh)$ as a contraction, then\\
(i) For $\alpha\in(0, \frac{1}{2}]$:
$$\varphi(w(A^{\alpha}XB^{1-\alpha}))\leq \|X\|w(\alpha\varphi(A)+(1-\alpha)\varphi(B)).$$ \\
(ii) For $\alpha\in [0, 1]$:
$$\varphi(w(A^{\alpha}XB^{1-\alpha}))\leq \frac{\|X\|}{2}w(\varphi(A^{2\alpha})+\varphi(B^{2(1-\alpha)})).$$
\end{rem}
In the following we establish an upper bound of the numerical radius of finite sum of operators of the form $\displaystyle\sum_{i=1}^{n}A_i^*X_iB_i$.
\begin{thm}{\label{thmsumaxb}}
Let $A_k, B_k, X_k\in \clb(\clh)$ for $k=1, 2,\ldots, n$ and $\varphi$ be an Orlicz function. Suppose that $f, g: [0, \infty)\rightarrow [0, \infty)$ be two continuous functions such that $f(t)g(t)=t$ for all $t\in [0, \infty)$. Then the following inequality holds true:
\begin{align*}
\varphi(w(\displaystyle\sum_{k=1}^{n}A_k^*X_kB_k))& \leq \frac{1}{2n}\displaystyle\sum_{k=1}^{n}\|\varphi(nB_k^*f^2(|X_k|)B_k)+\varphi(nA_k^*g^2(|X_k^*|)A_k)\|.
\end{align*}
\end{thm}
\begin{proof}
Let $x\in \clh$ with $\|x \|=1$. Then we get
\begin{align*}
&\varphi(|\langle \displaystyle\sum_{k=1}^{n}(A_k^*X_kB_k)x, x \rangle|)\\
& \leq  \varphi(\displaystyle\sum_{k=1}^{n}|\langle X_kB_kx, A_kx \rangle|)\\
& \leq \varphi(\displaystyle\sum_{k=1}^{n}\|f(|X_k|)B_kx\|\|g(|X_k^*|)A_kx\|)~~(\mbox{Lemma}~ \ref{kitt88})\\
& \leq \varphi(\displaystyle\sum_{k=1}^{n}\langle f^2(|X_k|)B_kx, B_kx \rangle^{\frac{1}{2}}\langle g^2(|X_k^*|)A_kx, A_kx \rangle^{\frac{1}{2}})
\end{align*}
\begin{align*}
&\leq \varphi(\displaystyle\sum_{k=1}^{n}\frac{1}{2}(\langle B_k^*f^2(|X_k|)B_kx, x \rangle+ \langle A_k^*g^2(|X_k^*|)A_kx, x \rangle))\\
& \leq \frac{1}{2n}\displaystyle\sum_{k=1}^{n}(\varphi(\langle n B_k^*f^2(|X_k|)B_kx, x \rangle)+\varphi(n\langle A_k^*g^2(|X_k^*|)A_kx, x \rangle)\\
& \leq \frac{1}{2n}\displaystyle\sum_{k=1}^{n}\langle \varphi(nB_k^*f^2(|X_k|)B_k)x, x \rangle + \langle\varphi(nA_k^*g^2(|X_k^*|)A_k)x, x \rangle ~(\mbox{Lemma}~ \ref{opjensen}).
\end{align*}This completes the proof.
\end{proof}
Recently, a sharp upper bound of numerical radius inequality (\ref{kitt05}) is obtained by Bhunia and Paul (see Theorem 2.11, \cite{BHUKAL}). In fact, for $r\geq 1$, and $A_k, B_k, X_k\in \clb(\clh)$ they have proved the following inequality:
\begin{align}{\label{bhukalineq21}}
w^r(\displaystyle\sum_{k=1}^{n}A_k^*X_kB_k)& =\frac{n^{r-1}}{\sqrt{2}}w\Big(\displaystyle\sum_{k=1}^{n}((B_k^*f^2(|X_k|)B_k)^r+ i (A_k^*g^2(|X_k^*|)A_k))^r\Big).
\end{align}
Next we present a general numerical radius inequality, which extends the above inequality (\ref{bhukalineq21}).
\begin{thm}{\label{thmbhukal}}
Let $A_k, B_k, X_k\in \clb(\clh)$ for $k=1, 2,\ldots, n$. Suppose that $f, g: [0, \infty)\rightarrow [0, \infty)$ be two continuous functions such that $f(t)g(t)=t$ for all $t\in [0, \infty)$. Then for an Orlicz function $\varphi$ the following inequality holds:
\begin{align*}
\varphi\Big(w\Big(\displaystyle\sum_{k=1}^{n}{A_k^*X_kB_k}\Big)\Big)&\leq \frac{1}{\sqrt{2}n}w\Big(\displaystyle\sum_{k=1}^{n}(\varphi(nB_k^*f^2(|X_k|)B_k)+ i \varphi(nA_k^*g^2(|X_k^*|)A_k))\Big).
\end{align*}
\end{thm}
\begin{proof}
Let $x\in \clh$ with $\|x\|=1$, and $A_k, B_k, X_k\in \clb(\clh)$ for $k=1, 2,\ldots, n$. Using the similar steps used in previous Theorem \ref{thmsumaxb}, we obtain the following expression:
\begin{align*}
&\varphi(|\langle \displaystyle\sum_{k=1}^{n}(A_k^*X_kB_k)x, x \rangle|)\\
& \leq \frac{1}{2n}\displaystyle\sum_{k=1}^{n}\langle \varphi(nB_k^*f^2(|X_k|)B_k)x, x \rangle + \langle\varphi(nA_k^*g^2(|X_k^*|)A_k)x, x \rangle\\
& \leq \frac{1}{\sqrt{2}n}\Big|\displaystyle\sum_{k=1}^{n}\langle \varphi(nB_k^*f^2(|X_k|)B_k)x, x \rangle + i \displaystyle\sum_{k=1}^{n}\langle\varphi(nA_k^*g^2(|X_k^*|)A_k)x, x \rangle \Big|\\
& ~~~~~~(\mbox{since}~ |p+q| \leq \sqrt{2}|p+iq|~\mbox{for~ all~} p, q\in \mathbb{R})\\
& = \frac{1}{\sqrt{2}n} \Big|\langle \displaystyle\sum_{k=1}^{n}\{\varphi(nB_k^*f^2(|X_k|)B_k)+ i\varphi(nA_k^*g^2(|X_k^*|)A_k)\}x, x\rangle \Big|\\
& \leq \frac{1}{\sqrt{2}n}w\Big(\displaystyle\sum_{k=1}^{n}(\varphi(nB_k^*f^2(|X_k|)B_k)+  i \varphi(nA_k^*g^2(|X_k^*|)A_k))\Big).
\end{align*}This proves the desired inequality.
\end{proof}

In operator theory, and its allied areas it is required to determine the bounds of numerical radius of operator like $ATB+CSD$ for $A, B, C, D, S, T\in \clb(\clh)$. The upper bound of $w(ATB+CSD)$ was established by Kittaneh \cite{KITTSM05}:
\begin{align}{\label{ineqkitt05}}
w(ATB+CSD)&\leq \frac{1}{2}\|A|T^*|^{2(1-\alpha)}A^*+  B^*|T|^{2\alpha}B + C|S^*|^{2(1-\alpha)}C^*+D^*|S|^{2\alpha}D\|,
\end{align} where $\alpha\in [0, 1]$. An attempt has been taken to establish a general version of inequality (\ref{ineqkitt05}), which includes power type numerical radius inequality as well as provides new inequalities for a particular $\varphi$. Let us begin with the following statement:
\begin{thm}{\label{thmkitt05}}
Suppose that $A, B, C, D, S, T\in \clb(\clh)$ and $\varphi$ be an Orlicz function. Moreover, let $f, g: [0, \infty)\rightarrow [0, \infty)$ be two continuous functions such that $f(t)g(t)=t$ for all $t\in [0, \infty)$. Then we have
\begin{align*}
&\varphi(w(ATB+CSD))\\
&\leq \Big\|\displaystyle\int_{0}^{1}\varphi(t({B^*}f^2(|T|)B + Ag^2(|T^*|){A^*})+(1-t)({D^*}f^2(|S|)D+Cg^2(|S^*|){C^*}))dt\Big\|\\
&\leq \frac{1}{2}\|\varphi(Ag^2(|T^*|)A^*+ B^*f^2(|T|)B) + \varphi(Cg^2(|S^*|)C^* + D^*f^2(|S|)D)\|.
\end{align*}
\end{thm}
\begin{proof}
Let $x\in \clh$ with $\|x\|=1$. Applying the AM-GM inequality and convexity of $\varphi$, we get
\begin{align*}
&\varphi(|\langle (ATB+CSD)x, x \rangle|)\\
& \leq \varphi(|\langle TBx, A^*x \rangle|+ |\langle SDx, C^*x \rangle|)\\
& \leq \varphi(\langle f^2(|T|)Bx, Bx \rangle^{\frac{1}{2}}\langle g^2(|T^*|)A^*x, A^*x \rangle^{\frac{1}{2}} \\
& ~~~~\hspace{0.25cm}+ \langle f^2(|S|)Dx, Dx \rangle^{\frac{1}{2}}\langle g^2(|S^*|)C^*x, C^*x \rangle^{\frac{1}{2}})~~(\mbox{Lemma}~ \ref{kitt88})\\
&\leq \varphi(\frac{1}{2}(\langle B^*f^2(|T|)Bx, x \rangle+ \langle Ag^2(|T^*|)A^*x, x \rangle)+ \frac{1}{2}(\langle D^*f^2(|S|)Dx, x \rangle + \langle Cg^2(|S^*|)C^*x, x \rangle)).
\end{align*}
Using inequality (\ref{HERMITEHAD}), we get
\begin{align*}
&\varphi(|\langle (ATB+CSD)x, x \rangle|)\\
& \leq \displaystyle\int_{0}^{1}\varphi\Big(t\langle \{Ag^2(|T^*|)A^*+ B^*f^2(|T|)B\}x, x\rangle+(1-t)\langle \{Cg^2(|S^*|)C^* + D^*f^2(|S|)D\}x, x \rangle \Big)dt\\
& \leq \frac{1}{2}\Big(\big\langle \big\{\varphi(B^*f^2(|T|)B+Ag^2(|T^*|)A^*)
+\varphi(D^*f^2(|S|)D+Cg^2(|S^*|)C^*)\big\}x, x \big\rangle\Big).
\end{align*}
Taking supremum over $\|x\|=1$ on both sides of the above, one gets the desired inequality. This proves the theorem.
\end{proof}
\begin{cor}Choose $f(t)=t^\alpha$ and $g(t)=t^{1-\alpha}$ for $t\geq 0$ and $\alpha\in [0, 1]$. Suppose that $\varphi(t)=t$ for $t\geq 0$. Then we have Kittaneh inequality \cite{KITTSM05} as below:
\begin{align*}
w(ATB+CSD)
&\leq \frac{1}{2}\|A|T^*|^{2(1-\alpha)}A^*+  B^*|T|^{2\alpha}B + C|S^*|^{2(1-\alpha)}C^*+D^*|S|^{2\alpha}D\|.
\end{align*}
\end{cor}
Now it is planned to establish a new numerical radius inequality in a more general setting. Consequently, we derive several recent improved numerical radius inequalities in one frame.
\begin{thm}{\label{thmphiwa2}}
Let $A\in \clb(\clh)$, and $\varphi$ be an Orlicz function. Then
\begin{align}{\label{ineqphiwa2}}
&\varphi(w^2(A))\\
& \leq \alpha\Big\|\frac{1}{4}(\varphi(|A|^{2})+  \varphi(|A^*|^{2}))+ \frac{1}{2}\varphi(|\mathfrak{R}(|A||A^*|)|)\Big\|+ \frac{(1-\alpha)}{2}\Big\|\displaystyle\int_{0}^{1}\varphi(t|A|^{2}+(1-t)|A^*|^{2})dt \Big\|\nonumber\\
& \hspace{0.5cm}+\frac{1-\alpha}{2}\varphi(w(A^2))\nonumber\\
&\leq \frac{1}{4}\|\varphi(|A|^{2})+  \varphi(|A^*|^{2})\|+ \frac{\alpha}{2}\|\varphi(|\mathfrak{R}(|A||A^*|)|)\|+ \frac{1-\alpha}{2}\varphi(w(A^2)),\nonumber
\end{align}where $\alpha\in [0, 1]$.
\end{thm}
\begin{proof}
Let $A\in \clb(\clh)$ and $x\in \clh$ with $\|x\|=1$. Then we have
\begin{align*}
&\varphi(|\langle Ax, x \rangle|^2)\\
& = \varphi(\alpha|\langle Ax, x \rangle|^2+ (1-\alpha)|\langle Ax, x \rangle|^2)\\
&\leq \alpha \varphi((\langle |A|x, x \rangle^{\frac{1}{2}} \langle |A^*|x, x \rangle^{\frac{1}{2}})^2) + (1-\alpha)\varphi(|\langle Ax, x \rangle|^2).
\end{align*}
Now we have the following
\begin{align*}
& \alpha \varphi((\langle |A|x, x \rangle^{\frac{1}{2}} \langle |A^*|x, x \rangle^{\frac{1}{2}})^2)\\
& \leq \alpha\varphi\Big(\Big(\frac{\langle |A|x, x \rangle + \langle |A^*|x, x \rangle}{2}\Big)^2\Big)\\
& \leq \alpha\varphi\Big(\frac{\langle (|A|+|A^*|)^2x, x \rangle}{4}\Big)\\
& \leq \alpha\varphi\Big(\frac{\langle (|A|^2+|A^*|^2+ 2|\mathfrak{R}(|A||A^*|)|)x, x \rangle}{4}\Big)\\
&\leq \alpha\Big(\langle\big(\frac{1}{4}\varphi (|A|^2)+\frac{1}{4}\varphi(|A^*|^2)+ \frac{1}{2}\varphi(|\mathfrak{R}(|A||A^*|)|)\big)x, x \rangle\Big).
\end{align*}
On the other hand, by Theorem \ref{improvmkitt} we get
\begin{align}
&(1-\alpha)\varphi(|\langle Ax, x \rangle|^2)\nonumber\\
& \leq \frac{(1-\alpha)}{2}\Big\langle\displaystyle\int_{0}^{1} \{\varphi(t|A|^{2}+(1-t)|A^*|^{2})dt\}x, x\Big\rangle+ \frac{(1-\alpha)}{2}\varphi(|\langle A^2x, x \rangle|) \nonumber\\
& \leq  \frac{(1-\alpha)}{4}(\langle \varphi(|A|^{2})x, x \rangle + \langle \varphi(|A^*|^{2})x, x \rangle)+ \frac{(1-\alpha)}{2}\varphi(|\langle A^2x, x \rangle|)~~(\mbox{Lemma}~ \ref{opjensen}).\nonumber
\end{align}
Therefore, we have
\begin{align*}
&\varphi(|\langle Ax, x \rangle|^2)\\
&\leq \alpha\Big(\langle\big(\frac{1}{4}\varphi (|A|^2)+\frac{1}{4}\varphi(|A^*|^2)+ \frac{1}{2}\varphi(|\mathfrak{R}(|A||A^*|)|)\big)x, x \rangle\Big)\\
& \hspace{0.5cm}+\frac{(1-\alpha)}{2}\Big\langle\displaystyle\int_{0}^{1} \{\varphi(t|A|^{2}+(1-t)|A^*|^{2})dt\}x, x\Big\rangle+ \frac{(1-\alpha)}{2}\varphi(|\langle A^2x, x \rangle|)\\
&\leq \frac{1}{4}(\langle \varphi(|A|^2)x, x \rangle + \langle \varphi(|A^*|^2)x, x \rangle)+ \frac{\alpha}{2}(\langle \varphi(|\mathfrak{R}(|A||A^*|)|)x, x \rangle)+\frac{1-\alpha}{2}\varphi(|\langle A^2x, x\rangle|).
\end{align*}
This proves the desired inequality.
\end{proof}

An Orlicz function $\varphi$ is said to be \emph{sub-multiplicative} if $\varphi(uv)\leq \varphi(u)\varphi(v)$ for $u, v\geq 0$, holds. In our next result, we find an upper bound by using the \emph{sub-multiplicative} property of $\varphi$.
\begin{thm}{\label{thmphiwa2new}}
Let $A\in \clb(\clh)$, and Orlicz function $\varphi$ is sub-multiplicative. Then for $\alpha\in [0, 1]$
\begin{align}{\label{ineqphiwa2new}}
\varphi(w^2(A))&\leq \frac{1}{4}\|\varphi(|A|^{2})+  \varphi(|A^*|^{2})\|+ \frac{\alpha}{2}\|\varphi(|A|)\| \|\varphi(|A^*|)\|)+ \frac{1-\alpha}{2}\varphi(w(A^2)).
\end{align}
\end{thm}
\begin{proof}
Let $A\in \clb(\clh)$ and $x\in \clh$. We follow the similar steps as in Theorem \ref{thmphiwa2}. Then we have
\begin{align}{\label{ineqdiff}}
&\varphi(|\langle Ax, x \rangle|^2) \nonumber\\
& \leq \alpha\varphi\Big(\Big(\frac{\langle |A|x, x \rangle + \langle |A^*|x, x \rangle}{2}\Big)^2\Big) \nonumber\\
& ~~~~~+ \frac{1-\alpha}{4}(\langle \varphi(|A|^2)x, x \rangle + \langle \varphi(|A^*|^2)x, x \rangle)+ \frac{1-\alpha}{2}\varphi(|\langle A^2x, x \rangle|)
\end{align}
Now applying the convexity and sub-multiplicative property of $\varphi$, the first term of RHS of inequality (\ref{ineqdiff}) gives the following:
\begin{align*}
&\alpha\varphi\Big(\Big(\frac{\langle |A|x, x \rangle + \langle |A^*|x, x \rangle}{2}\Big)^2\Big)\\
& \leq \frac{\alpha}{4} \varphi\Big(\langle |A|x, x \rangle^2\Big) + \frac{\alpha}{4}\varphi\Big(\langle |A^*|x, x \rangle^2\Big) + \frac{\alpha}{2} \varphi\Big(\langle |A|x, x \rangle \langle |A^*|x, x \rangle\Big)\\
& \leq \frac{\alpha}{4}\Big\{\langle \varphi(|A|^2)x, x \rangle + \langle \varphi(|A^*|^2)x, x \rangle\Big\}+ \frac{\alpha}{2} \varphi\Big(\langle |A|x, x \rangle\Big) \varphi\Big(\langle |A^*|x, x \rangle\Big)\\
& \leq \frac{\alpha}{4}\Big\{\langle \varphi(|A|^2)x, x \rangle + \langle \varphi(|A^*|^2)x, x \rangle\Big\}+ \frac{\alpha}{2} \langle \varphi(|A|)x, x \rangle \langle \varphi(|A^*|)x, x \rangle.
\end{align*}
Hence by plugging this upper bound in inequality (\ref{ineqdiff}), and then taking supremum over $\|x\|=1$, we get the desired inequality.
\end{proof}

In the following result, we have obtained a generalized version of Theorem \ref{thmphiwa2}.

\begin{thm}{\label{thmphiwa2fin}}
Let $A_i\in \clb(\clh)$ for $i=1, 2,\ldots, n$. Then for Orlicz function $\varphi$ the following inequality holds:
\begin{align*}
&\varphi\Big(w^2\Big(\displaystyle\sum_{i=1}^{n}{A_i}\Big)\Big)\\
& \leq \frac{\alpha}{n}\Big\|\displaystyle\sum_{i=1}^{n}\frac{1}{4}(\varphi(|nA_i|^{2})+  \varphi(|nA_i^*|^{2}))+ \frac{1}{2}\varphi(|\mathfrak{R}(|A||A^*|)|)\Big\| \\
& \hspace{0.5cm}+\frac{(1-\alpha)}{2n}\Big\|\displaystyle\sum_{i=1}^{n}\displaystyle\int_{0}^{1}\varphi(t|nA_i|^{2}+(1-t)|nA_i^*|^{2})dt \Big\|+\frac{1-\alpha}{2n}\displaystyle\sum_{i=1}^{n}\varphi(n^2w(A_i^2))\\
&\leq \frac{1}{4n}\|\displaystyle\sum_{i=1}^{n}(\varphi(|nA_i|^{2})+  \varphi(|nA_i^*|^{2}))\|+ \frac{\alpha}{2n}\|\displaystyle\sum_{i=1}^{n}\varphi\Big(\big|\mathfrak{R}(n^2|A_i||A_i^*|)\big|\Big)\|+ \frac{1-\alpha}{2n}\displaystyle\sum_{i=1}^{n}\varphi(n^2w(A_i^2)).
\end{align*}
\end{thm}
\begin{proof}
Let $x\in \clh$ with $\|x\|=1$, and $A_i\in \clb(\clh)$ for $i=1, 2,\ldots, n$. Then using the steps applied in Theorem \ref{thmphiwa2} and Lemma \ref{lembohr}, we have
\begin{align*}
&\varphi\Big(\Big|\Big\langle \Big(\displaystyle\sum_{i=1}^{n}{A_i}\Big) x, x \Big\rangle \Big|^2\Big)
= \varphi\Big(\Big|\displaystyle\sum_{i=1}^{n}\Big\langle {A_i} x, x \Big\rangle \Big|^2\Big)
\leq \varphi\Big(\Big(\displaystyle\sum_{i=1}^{n}\Big|\Big\langle {A_i} x, x \Big\rangle \Big|\Big)^2\Big)\\
& \leq \varphi\Big(n\displaystyle\sum_{i=1}^{n}\Big|\Big\langle {A_i} x, x \Big\rangle \Big|^2\Big) \leq \frac{1}{n}\displaystyle\sum_{i=1}^{n}\varphi\Big(n^2\Big|\Big\langle {A_i} x, x \Big\rangle \Big|^2\Big)\\
&\leq \frac{\alpha}{n}\displaystyle\sum_{i=1}^{n}\Big(\langle\big(\frac{1}{4}\varphi (|nA_i|^2)+\frac{1}{4}\varphi(|nA_i^*|^2)+ \frac{1}{2}\varphi(|\mathfrak{R}(n^2|A_i||A_i^*|)|)\big)x, x \rangle\Big)\\
& \hspace{0.5cm}+\frac{(1-\alpha)}{2n}\displaystyle\sum_{i=1}^{n}\Big\langle\displaystyle\int_{0}^{1} \{\varphi(t|nA_i|^{2}+(1-t)|nA_i^*|^{2})dt\}x, x\Big\rangle+ \frac{(1-\alpha)}{2n}\displaystyle\sum_{i=1}^{n}\varphi(n^2|\langle A_i^2x, x \rangle|)\\
& \leq \frac{1}{4n}( \displaystyle\sum_{i=1}^{n}\{\langle\varphi(|nA_i|^2)x, x \rangle + \langle \varphi(|nA_i^*|^2)\}x, x \rangle)+ \frac{\alpha}{2n}(\langle \displaystyle\sum_{i=1}^{n} \varphi\Big(\big|\mathfrak{R}(n^2|A_i||A^*_i|)\big|\Big)x, x \rangle)\\
& ~~~~+\frac{1-\alpha}{2n}\displaystyle\sum_{i=1}^{n}\varphi(n^2|\langle A_i^2x, x\rangle|).
\end{align*}Hence the desired inequality.
\end{proof}
In the next result, another numerical radius inequality is obtained by using Hermite-Hadamard inequality (\ref{HERMITEHAD}) and Orlicz function $\varphi$.
\begin{thm}{\label{thmdiflook}}
Let $A\in \clb(\clh)$ and $\varphi$ be an Orlicz function. Then the following inequality holds:
\begin{align*}
&\varphi(w^2(A))\\
& \leq \frac{1}{2}\Big\|\displaystyle\int_{0}^{1} \varphi(t|A|^{2}+(1-t)|A^*|^{2})dt \Big \|+ \frac{1}{2}\varphi(w(|A||A^*|))\\
&\leq \frac{1}{4}\|(\varphi(|A|^{2})+  \varphi(|A^*|^{2}))\|+ \frac{1}{2}\varphi(w(|A||A^*|)).
\end{align*}
\end{thm}
\begin{proof}
Suppose that $x\in \clh$ with $\|x\|=1$. Then one gets
\begin{align*}
&\varphi(|\langle Ax, x \rangle|^2)\\
&\leq \varphi(\langle |A|x, x \rangle \langle |A^*|x, x \rangle)\\
& = \varphi(\langle |A^*|x, x \rangle \langle x, |A|x \rangle)\\
&\leq \varphi(\frac{1}{2}(\||A^*|x\| \||A|x\| + |\langle |A^*|x, |A|x \rangle|))\\
& \leq \frac{1}{2}\varphi(\langle |A^*|x, |A^*|x \rangle^{\frac{1}{2}}\langle |A|x, |A|x \rangle^{\frac{1}{2}})+\frac{1}{2}\varphi(|\langle |A||A^*|x, x \rangle|)\\
& \leq \frac{1}{2}\varphi(\frac{\langle |A^*|^2x, x \rangle + \langle |A|^2x, x \rangle}{2})+ \frac{1}{2}\varphi(|\langle |A||A^*|x, x \rangle|).
\end{align*}
Applying the Hermite-Hadamard inequality (\ref{HERMITEHAD}) in above, one immediately gets
\begin{align*}
&\varphi(|\langle Ax, x \rangle|^2)\\
& \leq \frac{1}{2}\Big\langle\displaystyle\int_{0}^{1} \{\varphi(t|A|^{2}+(1-t)|A^*|^{2})dt\}x, x\Big\rangle+ \frac{1}{2}\varphi(|\langle |A||A^*|x, x \rangle|)\\
& \leq  \frac{1}{4}(\langle \varphi(|A^*|^2)x, x \rangle + \langle \varphi(|A|^2)x, x \rangle)+ \frac{1}{2}\varphi(|\langle |A||A^*|x, x \rangle|).
\end{align*}
This proves the desired inequality.
\end{proof}

\section{Applications to block matrices}

Let $P, Q\in \clb(\clh)$. Then we compute numerical radius of some $2\times2$ off-diagonal operator matrix in this section. In fact, we establish an generalized upper bound for $\varphi(w(A))$, where $A=\begin{bmatrix}O & P\\Q & O\end{bmatrix}\in \clb(\clh \oplus \clh)$. Our result is stated as below.
\begin{thm}{\label{thmblock}}
Let $P, Q \in \clb(\clh)$, and $\phi $ be an Orlicz function. If $A=\begin{bmatrix}O & P\\Q & O\end{bmatrix}\in \clb(\clh \oplus \clh)$, then
\begin{align*}{\label{ineqblock}}
\varphi(w^2(A))& \leq \frac{1}{4}\Big\|\begin{bmatrix}\varphi(|Q|^2)+ \varphi(|P^*|^2)& O\\O & \varphi(|P|^2)+ \varphi(|Q^*|^2)\end{bmatrix}\Big\| \nonumber\\
&~~~+\frac{\alpha}{2}\Big\|\begin{bmatrix}\varphi(|\mathfrak{R}(|Q||P^*|)|)& O\\O & \varphi(|\mathfrak{R}(|P||Q^*|)|)\end{bmatrix}\Big\|\nonumber\\
& ~~~+ \frac{1-\alpha}{2} \varphi\Big(w\Big(\begin{bmatrix}PQ & O\\O & QP\end{bmatrix}\Big)\Big).
\end{align*}
\end{thm}
\begin{proof}
Suppose that $\phi $ is an Orlicz function and
$A=\begin{bmatrix}O & P\\Q & O\end{bmatrix}\in \clb(\clh \oplus \clh)$.
Then we have \\
\begin{center}
$\varphi(|A|^2)=\varphi(A^*A)=\begin{bmatrix}\varphi(|Q|^2) & O\\ O & \varphi(|P|^2)\end{bmatrix}$ and $\varphi(|A^*|^2)=\varphi(AA^*)=\begin{bmatrix}\varphi(|P^*|^2) & O\\ O & \varphi(|Q^*|^2)\end{bmatrix}$.
\end{center}
Further, since
\begin{center}
$\mathfrak{R}(|A||A^*|)=\begin{bmatrix}\mathfrak{R}(|Q||P^*|) & O\\ O & \mathfrak{R}(|P||Q^*|)\end{bmatrix}$, \\
we get $\varphi(|\mathfrak{R}(|A||A^*|)|)=\begin{bmatrix}\varphi(|\mathfrak{R}(|Q||P^*|)|) & O\\ O & \varphi(|\mathfrak{R}(|P||Q^*|)|)\end{bmatrix}$.
\end{center}
Therefore, for the operator $A=\begin{bmatrix}O & P\\Q & O\end{bmatrix}$, the desired inequality follows immediately from Theorem \ref{thmphiwa2}.
\end{proof}
We now have the following  remark:
\begin{rem}(see Remark 2.6, \cite{BHUKAL21OCT})
Let $P, Q\in \clb(\clh)$ and $A=\begin{bmatrix}O & P\\Q & O\end{bmatrix}\in \clb(\clh \oplus \clh)$. Then for $\varphi(t)=t$ for $t\geq 0$, and $\alpha=1$, one gets
\begin{align*}
w^2(A)& \leq \frac{1}{4}\max\Big\{\||Q|^2+ |P^*|^2\|, \||P|^2+ |Q^*|^2\|\Big\}
+\frac{1}{2}\max\Big\{\||\mathfrak{R}(|Q||P^*|)|\|, \||\mathfrak{R}(|P||Q^*|)|\|\Big\}.
\end{align*}
\end{rem}

\section{Conclusions}
 In this paper we have expanded the improved numerical radius inequalities for class of operators defined on a Hilbert space by using the Orlicz function and Hermite-Hadamard inequality. We have also been able to generalize and extend many known results in this important area of research.



\begin{thebibliography}{99}

\bibitem{OMARKITT}
          {\sc A. Abu-Omar and F. Kittaneh},
          {\it Numerical radius inequalities for $n\times n$ operator matrices},
          {Linear Algebra Appl.}, {\bf468}(1) (2015), 18-26.
\bibitem{AXELSSON}
          {\sc O. Axelsson, H. Lu, B. Polman},
          {\it On Numerical Radius of Matrices and Its Application for Iterative Solution Methods (2nd ed.)},
          {Report 9308, Dept. of Mathematics, Catholic Univ. Nijmegen}, 1993.
\bibitem{BERG}
          {\sc C. A. Berger},
          {\it On the numerical range of powers of an operator},
          {Notices of Amer. Math. Soc.}, {\bf12} (No. 590) (1965).
\bibitem{BHUKAL}
          {\sc P. Bhunia and K. Paul},
          {\it New upper bounds for the numerical radius of Hilbert space operators},
          {Bull. Sci. Math.}, {\bf167} (2021), No. 102959.
\bibitem{BHUKAL21OCT}
          {\sc P. Bhunia and K. Paul},
          {\it Improvement of numerical radius inequalities},
          {Preprint}, {https://arxiv.org/pdf/2110.02505.pdf}.
\bibitem{BONSALL}
          {\sc F. F. Bonsall, J. Duncan},
          {\it Numerical Ranges II},
          {Cambridge University Press, ISBN 978-0-521-20227-5}, 1971.

\bibitem{BUZANO}
          {\sc M. L. Buzano},
          {\it Generalizzazione della diseguaglianza di Cauchy-Schwartz (Italian)},
          {Rend Sem Mat Univ E Politech Torino}, {1974 \bf31} (1971/73), 405-409.

\bibitem{Conway-Book}
{\sc J. B. Conway},
{\it A Course in Functional Analysis, 2nd edition},
{Graduate Texts in Mathematics}, vol. 96, Springer-Verlag, New York, 1990.



\bibitem{DRAGO}
         {\sc S. S. Dragomir},
         {\it Inequalities for the numerical radius of linear operators in Hilbert space},
         {Springer Briefs in Mathematics}, Springer, 2013.
\bibitem{HADDKITT}
          {\sc M. El-Haddad and F. Kittaneh},
          {\it Numerical radius inequalities for Hilbert space operators II},
          {Studia Math.}, {\bf182}(2) (2007), 133-140.
\bibitem{FONGHOLB}
          {\sc C. K. Fong and J. A. R. Holbrook},
          {\it Unitarily invariant operators norms},
          {Canad. J. Math.}, {\bf35} (1983), 274-299.
\bibitem{GOLDNUM}
          {\sc M. Goldberg, E. Tadmor and G. Zwas},
          {\it Numerical radius of positive matrices},
          {Linear Algebra Appl.}, {\bf12} (1975), 209-214.
\bibitem{GOLDON}
          {\sc M. Goldberg and E. Tadmor},
          {\it On the numerical radius and its applications},
          {Linear Algebra Appl.}, {\bf42} (1982), 263-284.
\bibitem{HADAMARD}
          {\sc J. Hadamard},
          {\it \'{E}tude sur les propri\'{e}t\'{e}s des fonctions enti\`{e}res et en particulier d'une fonction consid\'{e}r\'{e}e par Riemann},
          {J. Math. Pures. et Appl.}, {\bf58} (1893), 171-215.
\bibitem{HALMOS}
          {\sc P. R. Halmos},
          {\it A Hilbert space problem book},
          {Van Nostrand, New York}, {\bf} 1967.
\bibitem{HIRJKITT}
          {\sc O. Hirzallah, F. Kittaneh and K. Shebrawi},
          {\it Numerical radius inequalities for certain $2\times 2$ operator matrices},
          {Integral Equ. Oper. Theory}, {\bf71} (2011), 129-147.
\bibitem{HIRJKITT2}
          {\sc O. Hirzallah, F. Kittaneh and K. Shebrawi},
          {\it Numerical radius inequalities for commutators of Hilbert space operators},
          {Numer. Funct. Anal. Optim.}, {\bf32}(7) (2011), 739-749.
\bibitem{HOLB}
          {\sc J. A. R. Holbrook},
          {\it Multiplicative properties of the numerical radius in operator theory},
          {J. Reine Angew Math.}, {\bf237} (1969), 166-174.
\bibitem{HORN}
          {\sc R. A. Horn; C. R. Johnson},
          {\it Topics in Matrix Analysis},
          {Cambridge University Press}, 1991.
\bibitem{KATO}
          {\sc T. Kato},
          {\it Notes on some inequalities for linear operators},
          {Math. Ann.}, {\bf125} (1952), 208-212.
\bibitem{KITTSM03}
          {\sc F. Kittaneh},
          {\it Numerical radius inequality and an estimate for the numerical radius of the Frobenius companion matrix},
          {Studia Math.}, {\bf158}(1) (2003), 11-17.
\bibitem{KITTSM05}
          {\sc F. Kittaneh},
          {\it Numerical radius inequalities for Hilbert space operators},
          {Studia Math.}, {\bf168}(1) (2005), 73-80.
\bibitem{KITTNOTES}
          {\sc F. Kittaneh},
          {\it Notes on some inequalities for Hilbert space operators},
          {Publ. RIMS Kyoto Univ.}, {\bf24} (1988), 283-293.
\bibitem{LINDTZA}
          {\sc J. Lindenstrauss and L. Tzafriri},
          {\it Classical Banach Spaces I and II, Sequence spaces, Function spaces},
          {Springer-Verlag Berlin Heidelberg, Printed in Germany}, {\bf} (1996).
\bibitem{MAJEE}
          {\sc S. Majee, A. Maji, and A. Manna},
          {\it Numerical radius and Berezin number inequality},
          {J. Math. Anal. Appl.}, {\bf517}(1) (2023), 126566.
\bibitem{PEAR}
          {\sc C. Pearcy},
          {\it An elementary proof of the power inequality for the numerical radius},
          {Michigan Math. J.}, {\bf13}(3) (1966), 289-291.
\bibitem{PECA}
          {\sc J. Pe\v{c}ari\'{c}, T. Furuta, J. M. Hot and Y. Seo},
          {\it Mond-Pe\v{c}ari\'{c} method in operator inequalities, Inequalities for bounded self-adjoint operators on a Hilbert space},
          {Monograph in inequalities I}, 2nd Edition, Element, Zagreb, 2005.
\bibitem{SABA}
          {\sc M. Sababheh and H. R. Moradi},
          {\it More accurate numerical radius inequalities (I)},
          {Linear Multilinear Algebra}, {\bf69}(10) (2021), 1964-1973.

\bibitem{VASIKEC}
          {\sc M. P. Vasi\'{c} and D. J. Ke\^{c}ki\'{c}},
          {\it Some inequalities for complex numbers},
          {Math. Balkanica}, {\bf1} (1971), 282-286.
\bibitem{YAMA}
          {\sc T. Yamazaki},
          {\it On upper and lower bounds for the numerical radius and an equality condition},
          {Studia Math.}, {\bf178}(1) (2007), 83-89.





\end{thebibliography}
\end{document}